\documentclass{amsart}

\usepackage{graphicx}
\usepackage{amsthm,amsmath,amssymb}
\usepackage{enumitem}
\usepackage{microtype}
\usepackage{cutwin}
\usepackage[colorlinks,citecolor=blue,urlcolor=blue]{hyperref}

\numberwithin{equation}{section}

\theoremstyle{plain}
\newtheorem{theorem}{Theorem}[section]
\newtheorem{lemma}[theorem]{Lemma}

\newtheorem{corollary}[theorem]{Corollary}

\theoremstyle{definition}

\theoremstyle{remark}

\newtheorem{remark}[theorem]{Remark}

\renewcommand{\tilde}{\widetilde}

\renewcommand{\emptyset}{\varnothing}

\newcommand{\defeq}{:=}
\newcommand{\eqdef}{=:}
\newcommand{\om}{\omega}
\newcommand{\ee}{\mathrm{e}}
\newcommand{\e}{\varepsilon}
\newcommand{\eps}{\varepsilon}
\newcommand{\diam}{\textup{diam}}
\newcommand{\inte}{\textup{int}}
\newcommand{\R}{\mathbb{R}} 
\newcommand{\N}{\mathbb{N}} 

\renewcommand{\subset}{\subseteq}

\begin{document}

\title[Lattice self-similar sets are not Minkowski measurable]{Lattice self-similar sets on the real line are not Minkowski measurable}

\author{Sabrina Kombrink}
\address[Sabrina Kombrink]{Universit\"at zu L\"ubeck, Institut f\"ur Mathematik, Ratzeburger Allee 160, 23562 L\"ubeck, Germany}
\email{kombrink@math.uni-luebeck.de}
\author{Steffen Winter}
\thanks{This research was initiated and carried out while the authors were staying at the Institut Mittag-Leffler. The authors would like to thank the staff of the institute as well as the participants and organisers of the 2017 research programme \emph{Fractal Geometry and Dynamics} for the stimulating atmosphere, the excellent working environment, warm hospitality and financial support.}
\address[Steffen Winter]{Karlsruhe Institute of Technology, Department of Mathematics, Englerstr. 2, 76131 Karlsruhe, Germany}
\email{steffen.winter@kit.edu}

\begin{abstract}
	We show that any nontrivial self-similar subset of the real line that is invariant under a lattice iterated function system (IFS) satisfying the open set condition (OSC) is not Minkowski measurable. So far, this was only known for special classes of such sets. Thereby, we provide the last puzzle-piece in proving that under OSC a nontrivial self-similar subset of the real line is Minkowski measurable iff it is invariant under a nonlattice IFS, a 25-year-old conjecture.
\end{abstract}
\maketitle
\section{Introduction}
The Minkowski content was proposed by B.~B.~Mandelbrot \cite{Mandelbrot} as texture parameter for irregular sets (a measure of ``lacunarity''). Indeed, the Minkowski content can be used to understand the geometry of a fractal set beyond its (Hausdorff or Minkowski) dimension and in particular is a tool to distinguish between sets of the same dimension.
Besides its geometric relevance, the Minkowski content has attracted attention in connection with the Weyl-Berry conjecture concerning the distribution of the eigenvalues of the Laplacian on bounded domains $\Omega\subset\R^d$ with fractal boundaries. More precisely, M.~L.~Lapidus and C.~Pomerance showed in \cite{LapidusPomerance} that if $\Omega\subset\R$, then the second asymptotic term of the eigenvalue counting function can be expressed in terms of the Minkowski dimension and the Minkowski content of the boundary of $\Omega$, whenever these quantities exist.
However, although much progress has been made in recent years, in general it is not easy to decide, whether the Minkowski content of a given fractal set exists or not.

Assuming the open set condition (OSC) it was conjectured in \cite[Conjecture~3]{Lapidus_Dundee} (see also \cite[Section~5.2]{Gatzouras}) that a nontrivial self-similar set is Minkowski measurable (i.\,e.\ its Minkowski content exists, and is positive and finite) iff it arises from a nonlattice iterated function system (IFS). The progress in resolving this conjecture is as follows: Self-similar subsets of $\mathbb R$ generated from nonlattice IFS satisfying OSC were shown to be Minkowski measurable in \cite{Lapidus_Dundee,Falconer_Minkowski,Gatzouras} (the results of \cite{Gatzouras} hold for self-similar subsets of $\mathbb R^d$, too). For nontrivial self-similar subsets of $\mathbb R$ the converse, i.\,e.\ lattice sets are not Minkowski measurable, was shown in \cite{LapidusvanFrankenhuijsen} under additional assumptions. These assumptions address the geometric structure of the underlying feasible open set for the OSC and have been weakened in \cite{Euler,ErinSteffen}, see Section~\ref{sec:Minkowskimeasurability} for more details. 
However, up to now the conjecture remained unresolved for large classes of self-similar sets, see Section~\ref{sec:examples} for examples.

In the present article we fully remove the assumptions of \cite{LapidusvanFrankenhuijsen,Euler,ErinSteffen} and in this way provide the last puzzle-piece in proving that under OSC a nontrivial self-similar subset of $\mathbb R$ is Minkowski measurable iff it arises from a nonlattice IFS. This resolves the conjecture stated in \cite[Conjecture~3]{Lapidus_Dundee} and \cite[Section~5.2]{Gatzouras} for self-similar sets in $\mathbb R$.

The article is organised as follows.
After some preliminaries in Sections~\ref{sec:prelim:Minkowski} and \ref{sec:ss_OSC} we give a brief exposition of the key results from the literature in Section~\ref{sec:Minkowskimeasurability}. A class of self-similar sets for which Minkowski measurability had previously not been understood is discussed in Section~\ref{sec:examples}. Our main results are stated in Section~\ref{sec:results} and proved in Section~\ref{sec:proofs}.
We conclude by showing in Section~\ref{sec:equiv} that for sets in $\R$ the above-mentioned results from \cite{Euler,ErinSteffen} are equivalent.

\section{Preliminaries} \label{sec:prelim}
\subsection{Minkowski measurability}\label{sec:prelim:Minkowski}
Let $A$ denote a compact subset of the one-dimensio\-nal Euclidean space $(\mathbb R,\lvert\cdot\rvert)$ and let $\e>0$. Define the \emph{$\e$-parallel set} of $A$ to be $A_{\e}\defeq\{x\in\mathbb R\mid \inf_{a\in A}\lvert x-a\rvert\leq\e\}$. If the \emph{Minkowski dimension} $\dim_M(A)\defeq 1-\lim_{\e\searrow 0}\log(\lambda(A_{\e}))/\log(\e)$ exists, then we consider the \emph{rescaled volume function} $\e\mapsto\e^{\dim_M(A)-1}\lambda(A_{\e})$ defined on $(0,\infty)$, where $\lambda$ denotes the Lebesgue measure in $\R$.
If its limit as $\e\searrow 0$ exists, then we write
\begin{equation*}
	\mathcal M(A)\defeq\lim_{\e\searrow 0}\e^{\dim_M(A)-1}\lambda(A_{\e})
\end{equation*}
and call this value the \emph{Minkowski content} of $A$. If $\mathcal M(A)$ exists, and is positive and finite then we say that $A$ is \emph{Minkowski measurable}.

\subsection{Self-similar sets, open set condition, (non-)lattice and nontrivial}\label{sec:ss_OSC}
Let $\Phi\defeq\{\phi_1,\ldots,\phi_N\}$ with $N\in\mathbb N$, $N\geq2$ denote an iterated function system (IFS) consisting of similarities $\phi_j$ acting on $\mathbb R$.
Suppose that the IFS $\Phi$ satisfies the \emph{open set condition (OSC)}. That is, there exists a nonempty open set $O$ such that
\begin{equation}\label{eq:OSC}
	\phi_i(O)\subset O\quad\text{and}\quad \phi_i(O)\cap\phi_j(O)=\emptyset\quad \text{for}\ i,j\in\Sigma, i\neq j
\end{equation}
where $\Sigma\defeq\{1,\ldots,N\}$.  Any nonempty open set $O$ satisfying \eqref{eq:OSC} shall be called a \emph{feasible open set} for the IFS $\Phi$.
Let $r_i$ denote the similarity ratio of $\phi_i$. We say that $\Phi$ is \emph{lattice}, if the set $\{\log(r_i)\mid i\in\Sigma\}$ generates a discrete subgroup of $(\mathbb R,+)$. Otherwise, $\Phi$ is said to be \emph{nonlattice}. If $\Phi$ is lattice, then there exists a maximal $a>0$ such that $\{\log(r_i)\mid i\in\Sigma\}\subset a\mathbb Z$ and we call $r\defeq \ee^{a}$ the \emph{base} of $\Phi$.

The natural action of $\Phi$ on the class of subsets of $\mathbb R$ is defined via $\Phi A\defeq\bigcup_{i\in\Sigma} \phi_i A$ for $A\subset\mathbb R$.
By Hutchinson's theorem, there exists a unique nonempty compact set $F$ satisfying the invariance relation $\Phi F=F$. This set $F$ is called the \emph{self-similar set} associated with $\Phi$.
It is well-known that under OSC $\dim_M(F)$ exists and that $\#F>1$, where $\#$ denotes the cardinality.

$F$ is called \emph{nontrivial} if $\dim_M(F)<1$. Nontriviality of $F$ is equivalent to the assertion that any feasible open set $O$ satisfies $O\setminus\overline{\Phi O}\neq\emptyset$, see \cite[Corollary~5.6]{ErinSteffen_Tilings}. Here, $\overline{B}$  and $\partial {B}$ denote the topological closure and boundary of a set $B$ respectively.
A feasible open set $O$ for $\Phi$ is called \emph{strong}, if it has nonempty intersection with $F$, i.\,e.\ $O\cap F\neq \emptyset$.
Moreover, following \cite{ErinSteffen,W15} $O$ is called \emph{compatible}, if $\partial O\subset F$. (Notice, in \cite{ErinSteffen_Tilings} $O$ is called \emph{compatible} if $\partial\overline{O}\subset F$, which is a weaker condition on $O$.)
Let $\pi_F$ denote the \emph{metric projection} onto $F$, which is defined on the set of points $x\in \R$ with a unique nearest neighbour $y$ in $F$ by $\pi_F(x)\defeq y$. The set $O$ is said to satisfy the \emph{projection condition} if $\phi_i
O\subset\overline{\pi_F^{-1}(\phi_i F)}$ for $i\in\Sigma$.

\subsection{Known results on Minkowski measurability of self-similar sets in $\R$} \label{sec:Minkowskimeasurability}
Let $F\subset\R$ be the self-similar set of an IFS $\Phi$ as defined in Section~\ref{sec:ss_OSC} and let $I$ denote the interior of the convex hull of $F$, that is, $\overline{I}$ is the smallest closed interval containing $F$. Note that since $F$ is not a singleton, $I$ is nonempty.
\begin{theorem}\label{thm:literature}
	Suppose that $\Phi$ satisfies OSC.
	\begin{enumerate}
		\item\label{it:Falconer} \cite{Lapidus_Dundee,Falconer_Minkowski} If $\Phi$ is nonlattice and $\phi_i (\overline{I})\cap \phi_j(\overline{I})=\emptyset$ for $i\neq j$ (i.\,e.\ the \emph{strong separation condition} is satisfied), then $F$ is Minkowski measurable.
		\item\label{it:Gatzouras} \cite{Gatzouras} If $\Phi$ is nonlattice, then $F$ is Minkowski measurable.
		\item\label{it:LvF} \cite{LapidusvanFrankenhuijsen} If  $\Phi$ is lattice,
		$F$ is nontrivial and $I$ is a feasible open set for $\Phi$, then $F$ is not Minkowski measurable.
		\item\label{it:Euler} \cite{Euler} If  $\Phi$ is lattice,
		$F$ is nontrivial and $\Phi^m I$ is a feasible open set for $\Phi$ for some $m\in\mathbb N_0$, then $F$ is not Minkowski measurable.
		\item\label{it:ErinSteffen} \cite{ErinSteffen} Assume existence of a strong feasible open set $O$ for $\Phi$ that satisfies the projection condition and that allows for a finite partition of $(0,\infty)$ so that $\e\mapsto\lambda (F_{\e}\cap (O\setminus\Phi O))$ is polynomial on each partition interval. If $\Phi$ is lattice and $F$ is nontrivial, then $F$ is not Minkowski measurable.
	\end{enumerate}
\end{theorem}
We point out that the results of \cite{Gatzouras,ErinSteffen} stated above in \ref{it:Gatzouras} and \ref{it:ErinSteffen} hold in arbitrary dimension. For self-similar subsets of $\mathbb R$ the assumptions in \ref{it:Euler} and \ref{it:ErinSteffen} are equivalent, which we prove below in Theorem~\ref{thm:KK15equivKPW16}.
To clarify that there exist lattice self-similar sets which are not covered by the results \ref{it:LvF}--\ref{it:ErinSteffen}, we now discuss some examples with more complicated feasible open sets.

\subsection{Self-similar sets with complicated feasible open sets}\label{sec:examples}
Let $A>1$ and let $\mathcal D\defeq\{d_1,\ldots,d_N\}\subset \mathbb R$ be a digit set. Define similarities $\phi_j$ acting on $\mathbb R$ by $\phi_j(x)=(x+d_j)/A$ for $j\in\{1,\ldots,N\}$. Further, let
\begin{equation*}
	\mathcal D_1\defeq\mathcal D,\quad
	\mathcal D_n\defeq \mathcal D + A\mathcal D_{n-1},\ n\geq 2\quad\text{and}\quad
	\mathcal D_{\infty}\defeq\textstyle{\bigcup_{n=1}^{\infty}}\mathcal D_n.
\end{equation*}
By \cite[Theorem~4.4]{HeLau} the IFS $\Phi\defeq\{\phi_1,\ldots,\phi_N\}$ satisfies OSC iff $\mathcal D_{\infty}$ is uniformly discrete and $\#\mathcal D_k=N^k$ for all $k\geq 1$. ($\mathcal D_{\infty}$ is \emph{uniformly discrete} if there exists $r>0$ so that $\lvert x-y\rvert\geq r$ for all $x\neq y\in\mathcal D_{\infty}$.)
Thus, if one chooses $A,d_1,\ldots,d_N$ to be nonnegative integers, then OSC is satisfied iff $d_i\neq d_j (\textup{mod}\,A)$ for $i\neq j$. Depending on the choice of $A$ and $\mathcal D$ feasible open sets can be rather complicated. E.\,g.\ for the IFS $\Phi$ given by $N=3$, $A=4$, $d_1=0$, $d_2=1$ and $d_3=6$ OSC is satisfied but the assumptions of Theorem~\ref{thm:literature}\ref{it:LvF}--\ref{it:ErinSteffen} are violated, which can be seen as follows. For \ref{it:LvF} and \ref{it:Euler} we provide a proof in the next paragraph. The statement for \ref{it:ErinSteffen} then directly follows from the equivalence of \ref{it:Euler} and \ref{it:ErinSteffen} which we prove in Theorem~\ref{thm:KK15equivKPW16} below.

Fix $m\in\mathbb N_0$ and let $U\defeq\Phi^m I$, where $I=(0,2)$ in this example. We claim that $U$ is not feasible for $\Phi$. Without loss of generality we can assume that $m$ is odd, i.\,e.\ $m=2k+1$ for some $k\in\N_0$, since feasibility of $\Phi^{m} I$ would imply feasibility of $\Phi^{m+1} I$.
Writing $\phi_{\om}\defeq\phi_{\om_1}\circ\cdots\circ\phi_{\om_n}$ for $\om\in\Sigma^n$ the claim directly follows from
\begin{equation}\label{eq:ex:nonempty}
	\phi_1(U)\cap\phi_2(U)
	\supset \phi_1\left( \phi_3\phi_{23}^k(I) \right) \cap \phi_2\left( \phi_{23}^k\phi_2(I) \right)
	\neq\emptyset,
\end{equation}
which we now prove:
First observe that $\phi_1(0)=0$, $\phi_3(2)=2$ and $\phi_{23}(2/3)=2/3$.
Second, note that for the left endpoints of the intervals $\phi_{13}\phi_{23}^k(I)$ and $\phi_2\phi_{23}^k\phi_2(I)$ we have
\begin{align*}
	\phi_{13}\phi_{23}^k(0) - \phi_2\phi_{23}^k\phi_2(0)
	&= \phi_{13}\left( \phi_{23}^k\left(\tfrac{2}{3}\right)-\tfrac{2}{3}\left(\tfrac{1}{4}\right)^{2k}\right)
	- \phi_2\left(\phi_{23}^k\left(\tfrac{2}{3}\right)-\tfrac{5}{12}\left(\tfrac{1}{4}\right)^{2k} \right)\\
	&=\left(\tfrac{1}{4} \right)^{2k+2}
	>0.
\end{align*}
Third, the intervals $\phi_{13}\phi_{23}^k(I)$ and $\phi_2\phi_{23}^k\phi_2(I)$ both have length $2\cdot(1/4)^{2k+2}$. Therefore, they must overlap in an interval of length $(1/4)^{2k+2}$, showing \eqref{eq:ex:nonempty}.

\renewcommand\windowpagestuff{%
	\centering
	\includegraphics[width=0.75\linewidth]{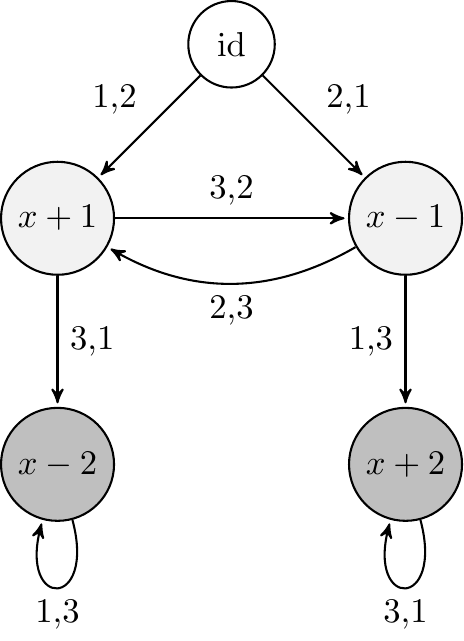}
}
\opencutright
\begin{remark}
	\begin{cutout}{12}{0.65\textwidth}{0.1\linewidth}{12}
	Indeed, in the above example, any feasible open set necessarily has an infinite number of connected components, disqualifying, in particular, the sets $\Phi^m I$. This was pointed out to us by Christoph Bandt, whom we wish to thank for sharing the following arguments with us: \\ \indent
	The \emph{dynamical boundary} of $F$ associated with $\Phi$ is the set $\textup{db}(F)\defeq\bigcup_{h} F\cap h F$, where the union is taken over all \emph{neighbour maps} $h$, i.\,e.\ maps of the form $h=\phi_u^{-1}\phi_{\om}$, where $u,\om\in\Sigma^n$ for some $n\in\mathbb N$ are so that $\phi_u(F)\cap\phi_{\om}(F)\neq\emptyset$ and $u_1\neq\om_1$. 
	When $x\in\textup{db}(F)$ then $\phi_u(x)\in\phi_u(F)\cap\phi_{\om}(F)$. Thus, any feasible open set $O$ for $\Phi$ may not intersect $\textup{db}(F)$. On the other hand, $\textup{db}(F)\subseteq F$ and whence $\textup{db}(F)\subseteq\overline{O}$. Therefore, $\textup{db}(F)\subseteq\partial{O}$.
	Now, if the dynamical boundary has infinite cardinality (which is the case here, see below), then $O$ necessarily has infinitely many connected components. \\ \indent
	In \cite{BandtMesing} general statements were obtained to determine the cardinality of the dynamical boundary of a limit set of a graph-directed system via neighbour graphs.
	The neighbour graph associated to the present example is depicted to the right. Its root is the identity and its vertices are the neighbour maps. ``An arrow with label $i,j$ is drawn from vertex $h$ to vertex $\overline{h}$ if $\overline{h}=\phi_i^{-1}h\phi_j$ for two marks $i,j\in\Sigma$. We keep only those arrows which correspond to proper neighbors, that is $\phi_i(F)\cap h\phi_j(F)\neq\emptyset$." \cite{Bandt_pinwheel}. Using the terminology from \cite{BandtMesing} the light shaded vertices are intermediate and the dark shaded ones are terminal. 
	According to \cite[Theorem 7]{BandtMesing} the terminal and intermediate vertices correspond to subsets of the dynamical boundary with cardinality one and countably infinite respectively. Thus, $\textup{db}(F)$ is countably infinite here.
\end{cutout}
\end{remark}

\section{Main results}\label{sec:results}

\begin{theorem}\label{thm:main}
	If $F$ is a nontrivial self-similar set in $\R$ generated by a lattice IFS $\Phi$ satisfying OSC, then $F$ is not Minkowski measurable.
\end{theorem}

Together with Theorem~\ref{thm:literature}~\ref{it:Gatzouras} we thus verify the conjecture of \cite[Conjecture~3]{Lapidus_Dundee} and \cite[Section~5.2]{Gatzouras} for self-similar sets in $\mathbb R$:

\begin{corollary}\label{cor:iff}
	Suppose that $F$ is a nontrivial self-similar set in $\R$ generated by an IFS $\Phi$ satisfying OSC. Then $F$ is Minkowski measurable iff $\Phi$ is nonlattice.
\end{corollary}

\begin{remark}
	The nontriviality condition, $\dim_M(F)<1$, is necessary in the statements of Theorem~\ref{thm:main}, Corollary~\ref{cor:iff} and Theorem~\ref{thm:literature}\ref{it:LvF}--\ref{it:ErinSteffen} and cannot be removed: The unit interval $X:=[0,1]$ has Minkowski dimension $\dim_M(X)=1$. It is the self-similar set associated with the lattice IFS $\{x\mapsto x/2,x\mapsto (x+1)/2\}$ acting on $\mathbb R$. However, its Minkowski content $\mathcal M(X)=\lim_{\e\to 0}(1+2\e)=1$ exists as a positive and finite value. Hence $X$ is Minkowski measurable. In fact, any self-similar set $F$ in $\R$ with $\dim_M F=1$ is Minkowski measurable, see e.\,g.\  \cite[Theorem~1.1(i)]{ErinSteffen}.
\end{remark}

A key ingredient in the proof of Theorem~\ref{thm:main} is the construction of a relatively simple strong feasible open set, see Theorem~\ref{thm:openset} below and its proof. With this set at hand we can deduce Minkowski non-measurability from \cite[Theorem~3.1]{ErinSteffen}, see Theorem~\ref{thm:subresult} below.
Define $\Sigma^*\defeq\bigcup_{n=0}^{\infty} \Sigma^n$ with $\Sigma^0\defeq\{\emptyset\}$, where $\emptyset$ denotes the empty word. Moreover, for $\om=(\om_1,\ldots,\om_n)\in\Sigma^*$ write $\phi_{\om}\defeq\phi_{\om_1}\circ\ldots\circ\phi_{\om_n}$ and let $\phi_{\emptyset}$ be the identity map.
\begin{theorem}\label{thm:openset}
	Let $F\subset\R$ be the self-similar set generated by an IFS $\Phi$ satisfying OSC. Then there exists a strong and compatible feasible open set $U$ for $\Phi$, i.\,e.\ one which satisfies $U\cap F\neq\emptyset$ and $\partial U\subset F$.

	What is more, there always exists such a set $U$ that can be generated from a finite union of elementary intervals $\phi_{u}(I)$: there exist  $m\in\mathbb N_0$ and $\Lambda\subset\Sigma^m$ such that
	\begin{equation}\label{eq:U}
		U_{\Lambda}\defeq\bigcup_{\om\in\Sigma^*}\phi_{\om}\bigcup_{u\in\Lambda}\phi_u(I)
	\end{equation}
	defines a strong and compatible feasible open set for $F$.
\end{theorem}

\begin{remark}\label{rem:U}
	For any $m\in\mathbb N_0$ and any nonempty $\Lambda\subset \Sigma^m$, the set $U_\Lambda$ in \eqref{eq:U} has nonempty intersection with $F$, since $I\cap F\neq\emptyset$. 
	Moreover, $U_{\Lambda}$ is compatible, because $\partial U_{\Lambda}\subset \overline{\bigcup_{\om\in\Sigma^*}\phi_{\om}\bigcup_{u\in\Lambda}\phi_u(\partial I)}\subset F$, where the last inclusion follows since $\partial I\subset F$ and $\phi_\om F\subset F$ for any $\om\in\Sigma^*$. However, it is not obvious that $U_\Lambda$ is a feasible open set and this is indeed only true for particular choices of $\Lambda$.
\end{remark}

\section{Proofs}\label{sec:proofs}

\subsection{Construction of a feasible open set $U_{\Lambda}$ -- Proof of Theorem~\ref{thm:openset}} Obviously, the first statement of the theorem follows from the second. In view of Remark~\ref{rem:U}, it therefore suffices to show that at least one of the sets $U_\Lambda$ (defined by \eqref{eq:U}) is feasible.
First observe that for any $m\in\N$ and any nonempty $\Lambda\subset\Sigma^m$ the set $U_\Lambda$  is nonempty and open, since $I$ has these properties, and that $\phi_i U_{\Lambda}\subset U_{\Lambda}$ for any $i\in\Sigma$. Therefore, all that remains to be shown is existence of a set $\Lambda\subset\Sigma^m$ for some $m\in\mathbb N_0$ such that $\phi_i(U_{\Lambda})\cap\phi_j(U_{\Lambda})=\emptyset$ for any $i\neq j\in\Sigma$. For this we adapt Schief's construction in \cite{Schief} of a strong feasible open set:
	
Let $r_{\om}$ denote the similarity ratio of $\phi_{\om}$ for $\om=(\om_1,\ldots,\om_n)\in\Sigma^*$. Note that $r_{\om}=r_{\om_1}\cdots r_{\om_n}$.
Fix $\e\in(0,1/6)$.
Schief showed \cite[proof of Theorem~2.1]{Schief} that there exists $\kappa\in\Sigma^*$ so that
\begin{equation*}
	O_{\kappa}
	\defeq
	\textstyle{\bigcup_{u\in\Sigma^*}}\phi_{u\kappa}\left(F_{\e}\right)
\end{equation*}
is a feasible open set for $\Phi$.	
As $r_i<1$ for all $i$, there is a minimal $k\in\mathbb N_0$ such that
\begin{equation*}
	\textstyle{\bigcup_{\om\in\Sigma^{k}}}\phi_{\om}I\subset F_{\e}.
\end{equation*}
Set $m\defeq k+\lvert\kappa\rvert$, where $\lvert\kappa\rvert$ denotes the length of $\kappa$, i.\,e.\ $\kappa\in\Sigma^{\lvert\kappa\rvert}$. Further, set $\Lambda\defeq\{\kappa\om\mid \om\in\Sigma^{k}\}$. Then $\emptyset\neq\Lambda\subset\Sigma^m$ and $U_{\Lambda}\subset O_{\kappa}$, whence $\phi_i(U_{\Lambda})\cap\phi_j(U_{\Lambda})=\emptyset$ for any $i\neq j\in\Sigma$.
This completes the proof of Theorem~\ref{thm:openset}.

\subsection{A criterion for Minkowski measurability.}\label{sec:criterion} In the proof of Theorem~\ref{thm:main} we will make use of a general Minkowski measurability criterion for self-similar sets in $\R^d$ (satisfying OSC) derived in \cite{ErinSteffen}. It is based on feasible open sets satisfying the projection condition and was obtained via classical renewal theory. We briefly restate a version of this criterion here, boiled down to our present one-dimensional setting.
Given $\Phi$, $O$ and $F$ as in Section~\ref{sec:prelim} we set
\begin{equation} \label{eq:Gandg}
	\Gamma\defeq O\setminus \Phi O\quad\text{and}\quad
	g\defeq \sup\{\inf_{y\in F}\lvert x-y\rvert\mid x\in\Gamma\}.
\end{equation}	

\begin{theorem}\cite[Theorem~3.1 and Corollary~3.2]{ErinSteffen}\label{thm:subresult} 
	 Let $F\subset\R$ be a nontrivial self-similar set generated by a lattice IFS $\Phi$ with base $r$. Suppose that $\Phi$ satisfies OSC with a strong feasible set $O$ satisfying the projection condition. Let $D:=\dim_M(F)$ and let $\Gamma$ and $g$ be defined as in \eqref{eq:Gandg}.
	 Define the function $p:(rg,g]\to\R$ by
  \begin{align}
  \label{eq:p}
    	p(\eps)
  	\defeq \eps^{D-1} \left[ \frac{\lambda(\Gamma)}{r^{D-1}-1} + \sum_{\ell = 0}^{\infty} r^{\ell(D-1)} \lambda(F_{r^{\ell}\eps}\cap \Gamma) \right].
  \end{align}
  Then $F$ is Minkowski measurable iff $p$ is constant on $(rg,g]$.
\end{theorem}

Note that the series in the definition of $p$ is uniformly convergent in $\e$, see \cite[proof of Theorem~3.1]{ErinSteffen}.
\begin{remark}\label{rmk:projection}
	It is easily seen that a feasible open set of the form $U_{\Lambda}$ given in \eqref{eq:U} satisfies the projection condition. In fact, any strong and compatible feasible open set $O$ satisfies the projection condition, see e.\,g.\ \cite[Remark~3.20]{W15}.
\end{remark}

\subsection{Minkowski non-measurability -- Proof of Theorem~\ref{thm:main}} \label{sec:proofs2}
Let $r$ denote the lattice base of $\Phi$ and let $U$ be a strong and compatible feasible open set for $\Phi$. Such a set $U$ exists due to Theorem~\ref{thm:openset}.
We want to apply  Theorem~\ref{thm:subresult} and note that all its assumptions are satisfied; in particular, the projection condition, see Remark~\ref{rmk:projection}. We infer that the set $F$ is Minkowski measurable iff the function $p$ defined in \eqref{eq:p} is constant.

In the following we will demonstrate that the properties of $U$ imply that $p$ cannot be constant. We pursue a proof by contradiction, whence assume that there exists $C>0$ so that $p(\e)=C$ or, equivalently,
\begin{equation}\label{eq:Widerspruch}
	L(\e)\defeq C\e^{1-D}-\frac{\lambda(\Gamma)}{r^{D-1}-1}
	=\sum_{\ell = 0}^{\infty} r^{\ell(D-1)} \lambda(F_{r^{\ell}\e}\cap \Gamma)
	\eqdef R(\e)
\end{equation}
for $\e\in(rg,g]$. 
Define $G\defeq U\setminus\overline{\Phi U}$.
Clearly, $G$ is open and $G\subset\Gamma$.
Moreover, $\lambda(\Gamma\setminus G)=0$, since $\Gamma\setminus G=U\cap\partial \Phi U\subset U\cap\Phi\partial U\subset F$ and $\dim_M(F)<1$. Therefore, $\lambda(F_{r^{\ell}\e}\cap \Gamma)=\lambda(F_{r^{\ell}\e}\cap G)$. As stated in Section~\ref{sec:ss_OSC}, nontriviality implies $G\neq\emptyset$. Hence $G$ has countably many connected components $G_j$, $j\in J$, each of which is an open interval.
Without loss of generality suppose that either $J=\mathbb N_0$ or $J=\{0,\ldots,n\}$ for some $n\in\mathbb N_0$. Write $\diam(G_j)$ for the diameter of $G_j$ and assume that the $G_j$ are ordered so that $\diam(G_{j-1})\geq\diam(G_{j})$ for all $j\in J\setminus\{0\}$. Since $\partial G_j\subset F$ and $G_j\cap F=\emptyset$, we have
\begin{equation}\label{eq:Gj}
	\lambda(F_{t}\cap G_j)=
	\begin{cases}
		2t &\colon 0<2t\leq\diam(G_j),\\
		\diam(G_j) & \colon 2t>\diam(G_j).
	\end{cases}
\end{equation}
For $\ell\in\mathbb N_0$, $j\in J$ define $f_{\ell,j}\colon(rg,g]\to\mathbb R$ by
\begin{equation*}
	f_{\ell,j}(\e)\defeq r^{\ell(D-1)}\lambda(F_{r^{\ell}\e}\cap G_j).
\end{equation*}
Then $R(\e)=\sum_{\ell=0}^{\infty}\sum_{j\in J} f_{\ell,j}(\e)$.
Let $f^{(-)}_{\ell,j}(\e)$ and $f^{(+)}_{\ell,j}(\e)$ denote the left and right derivatives of $f_{\ell,j}$ at $\e$ respectively. By \eqref{eq:Gj}, we have that
\begin{equation}\label{eq:allelj}
	f^{(-)}_{\ell,j}(\e)\geq f^{(+)}_{\ell,j}(\e)\geq 0\ \text{for}\ \ell\in\mathbb N_0,\ j\in J\ \text{and}\ \e\in(rg,g).
\end{equation}
In fact, since $\lambda(F_{\eps}\cap G_j)$ is piecewise linear with at most two different slopes, the derivative $f'_{\ell,j}(\e)$ of $f_{\ell,j}$ exist at any $\e\in(rg,g)$ except for at most one point.
\begin{lemma}\label{lem:derivativesuniform}
	The series $\sum_{\ell=0}^{\infty}\sum_{j\in J} f^{(\pm)}_{\ell,j}$ converge uniformly on $(rg,g)$.
\end{lemma}
\begin{proof}
	Let $N_{\e}(\ell)\defeq\#\{j\in J\mid 2 r^{\ell}\e\leq\diam(G_j)\}$.
	As remarked in Section~\ref{sec:criterion}, the series $\sum_{\ell = 0}^{\infty} r^{\ell(D-1)} \lambda(F_{r^{\ell}\e}\cap \Gamma)$ from \eqref{eq:p} is uniformly convergent in $\e$. Thus, there exists a sequence $(c_n)_n$ so that $\lim_{n\to \infty}c_n=0$ and so that for $\e\in(rg,g]$, $n\in\mathbb N$
	\begin{align*}
		c_n
		\geq 
		\sum_{\ell = n}^{\infty} r^{\ell(D-1)} \sum_{j\in J}\lambda(F_{r^{\ell}\e}\cap G_j)
		\geq \sum_{\ell = n}^{\infty} r^{\ell(D-1)} \sum_{j=0}^{N_{\e}(\ell)-1}
		2r^{\ell}\e
		= \e\sum_{\ell = n}^{\infty} 2r^{\ell D}N_{\e}(\ell).
	\end{align*}
	Since $f_{\ell,j}^{(-)}(\e)=0$ if $2r^{\ell}\e>\diam(G_j)$, this yields
	\begin{align*}
	\sum_{\ell=n}^{\infty}\sum_{j\in J} f^{(-)}_{\ell,j}(\e)
	= \sum_{\ell=n}^{\infty}\sum_{j=0}^{N_{\e}(\ell)-1} 2r^{\ell D}	
	=\sum_{\ell=n}^{\infty} 2r^{\ell D}N_{\e}(\ell)
	\leq \frac{c_n}{\e}
	\leq \frac{c_n}{rg}
	\end{align*}
	which proves uniform convergence of $\sum_{\ell=0}^{\infty}\sum_{j\in J} f^{(-)}_{\ell,j}$ and by \eqref{eq:allelj} also of the series $\sum_{\ell=0}^{\infty}\sum_{j\in J} f^{(+)}_{\ell,j}$.
\end{proof}

\begin{remark}
   Observe that $f_{\ell,j}$ are Kneser functions of order 1, i.\,e. they satisfy
\begin{align*}
  f_{\ell,j}(\mu b)-f_{\ell,j}(\mu a)\leq \mu (f_{\ell,j}(b)-f_{\ell,j}(a)),\
\end{align*}
for all $a,b\in(rg,g)$ with $a\leq b$ and any $\mu\geq 1$ such that $\mu b<g$. This can be checked directly, but it also follows from \cite[Lemma~5]{Stacho}, since the intervals $G_j$ are metrically associated with $F$ (meaning that for each point $x\in G_j$ there is a point $y\in F$ with $\lvert x-y\rvert=\inf_{a\in F}\lvert x-a\rvert$ such that the whole segment between $x$ and $y$ is contained in $G_j$) and therefore $\lambda(F_t\cap G_j)$ is a Kneser function of order 1 on $(0,\infty)$. Hence, the assertion of Lemma~\ref{lem:derivativesuniform} is a special case of \cite[Lemma~4]{Stacho}.
\end{remark}

In order to obtain a contradiction, we consider two cases:

\underline{\textsc{Case} 1:} \emph{There exist $\ell^*\in\mathbb N_0$, $j^*\in J$ and $x\in(rg,g)$ so that $f^{(-)}_{\ell^*,j^*}(x)\neq f^{(+)}_{\ell^*,j^*}(x)$}.\\
Equation \eqref{eq:Gj} implies
\begin{align}\label{eq:ljstar}
	 f^{(-)}_{\ell^*,j^*}(x)=2 r^{\ell^* D}>0=f^{(+)}_{\ell^*,j^*}(x)
\end{align}
Lemma~\ref{lem:derivativesuniform} shows that $R^{(\pm)}(x)$ exist and that $R^{(\pm)}(x) = \sum_{\ell=0}^{\infty}\sum_{j\in J}^{\infty} f^{(\pm)}_{\ell,j}(x)$.
With \eqref{eq:allelj} and \eqref{eq:ljstar} we thus obtain
\begin{align*}
	R^{(-)}(x) - R^{(+)}(x)
	&\geq f^{(-)}_{\ell^*,j^*}(x) - f^{(+)}_{\ell^*,j^*}(x)
	=2r^{\ell^* D}
	>0.
\end{align*}
Hence, as opposed to the function $L$, the function $R$ is not differentiable at $x$, contradicting \eqref{eq:Widerspruch}.

\underline{\textsc{Case} 2:} \emph{The derivative $f'_{\ell,j}$ exists on $(rg,g)$ for all $\ell\in\mathbb N_0$, $j\in J$.}\\
In this case, for any $j\in J$ there exists $k=k(j)\in\mathbb N_0$ so that $\diam (G_j)=2r^kg$, yielding
\begin{equation}\label{eq:derivativeflj}
	f'_{\ell,j}\equiv 2r^{\ell D}\ \text{on}\ (rg,g)\ \text{for all}\  \ell\geq k(j), \ \text{and}\ f'_{\ell,j}\equiv 0 \ \text{otherwise}.
\end{equation}
By Lemma~\ref{lem:derivativesuniform}, $R'$ exists and coincides with $\sum_{\ell=0}^{\infty}\sum_{j\in J}f'_{\ell,j}$ which by \eqref{eq:derivativeflj} is constant on $(rg,g)$. However, $L'(\e)=C(1-D)\e^{-D}$ which, due to the nontriviality of $F$ (and since $C>0$), is clearly not constant. Therefore, we obtain a contradiction to \eqref{eq:Widerspruch} also in the second case.
This completes the proof of Theorem~\ref{thm:main}.

\section{Equivalence of {\rm \ref{it:Euler}} and {\rm \ref{it:ErinSteffen}} of Theorem~\ref{thm:literature}}\label{sec:equiv}

During our discussions the question arose whether the classes of self-similar subsets of $\mathbb R$ covered by the assertions \ref{it:Euler} and \ref{it:ErinSteffen} of Theorem~\ref{thm:literature} are equivalent. The following statement gives an affirmative answer (irrespective of the IFS being lattice or nonlattice).
\begin{theorem}\label{thm:KK15equivKPW16}
	Let $\Phi$ be an IFS in $\R$ satisfying OSC such that the associated invariant set $F$ is nontrivial. Then the following assertions are equivalent.
	\begin{enumerate}
		\item \label{it:UnionElementaryOpen} $\Phi^m I$ is a feasible open set for $\Phi$ for some $m\in\mathbb N_0$.
		\item \label{it:pluriphaseequiv} There exists a strong feasible open set $O$ for $\Phi$ satisfying the projection condition that allows for a finite partition of $(0,\infty)$ so that \mbox{$\e\mapsto\lambda (F_{\e}\cap (O\setminus\Phi O))$} is polynomial on each partition interval.
	\end{enumerate}
\end{theorem}

\begin{proof}
To begin with, note that any feasible open set of the form $\Phi^m I$ allows for a finite partition of $(0,\infty)$ so that $\e\mapsto\lambda(F_{\e}\cap (\Phi^m I\setminus \Phi^{m+1}I))$ is polynomial on each partition interval. Therefore, \ref{it:UnionElementaryOpen} implies \ref{it:pluriphaseequiv}.

For the converse, suppose that $O$ is as in \ref{it:pluriphaseequiv}.
Consider $U\defeq \inte\left(\overline{O}\right)$, where $\inte$ denotes the topological interior. Then $U=\bigcup_{i\in E} U_i$ is a union of open intervals $U_i$ with the property that the distance between any two $U_i$ is strictly positive.
Let $\tilde E:=\{i\in E: U_i\cap I\neq\emptyset\}$. The key part of the proof is to show that
\begin{equation}\label{claimIII} 
	\#\tilde E<\infty. 
\end{equation}
Before proving \eqref{claimIII} we demonstrate that \eqref{claimIII} implies assertion \ref{it:UnionElementaryOpen}.
Since $O$ is a strong feasible open set for $\Phi$, so is $U$, which can be seen by contradiction.
Therefore, $F\subset \overline{U}$ and so, by \eqref{claimIII}, $F\subset \overline{U}\cap \overline{I}\subset\bigcup_{i\in \tilde E} \overline{U_i}$, which implies that there exists $m\in\mathbb N$ so that $\Phi^m\overline{I}\subset \bigcup_{i\in \tilde E} \overline{U_i}$ (simply choose $m$ large enough that, for any $w\in\Sigma^m$, $\diam(\phi_w\overline{I})$ is smaller than the minimal distance between the finitely many $U_i$). The property that the $U_i$ have positive distance to one another implies that $\Phi^m I\subset \bigcup_{i\in E} U_i$. From this inclusion it is easy to see that $\Phi^mI$ is feasible for $\Phi$, whence assertion \ref{it:UnionElementaryOpen} holds.

To verify \eqref{claimIII} let $c_1, \ldots,c_k\in (0,\infty)$ denote the partition points of the partition of $(0,\infty)$ associated with $O$. 
Let $\{H_j\}_{j\in J}$ denote the collection of connected components of $\inte (O\setminus\Phi O)$. Clearly, each $H_j$ is an open interval and it is easy to see that $H_j\cap F=\emptyset$.
We show \eqref{claimIII} in four steps. Our first one is to prove
\begin{enumerate}[label=(\Roman*),start=1]
  \item \label{it:claimI} $\# J<\infty$.
\end{enumerate}
For this, set $h_j\defeq\inf\{\e>0: H_j\subset F_\eps\}$.
Observe that $\e\mapsto\lambda (F_{\e}\cap H_j)$ is constant (and equal to $\lambda(H_j)$) for $\e>h_j$, and linear (and nonconstant) in a left neighborhood of $h_j$. In particular, the function $\e\mapsto \lambda (F_{\e}\cap H_j)$ is not differentiable at $h_j$ and so $h_j$ must be one of the partition points $c_{\ell}$. Next we show that for each of the finitely many $c_{\ell}$ the associated set $J_{\ell}\defeq\{j\in J\mid h_j=c_{\ell}\}$ is of finite cardinality: For $j\in J_{\ell}$ let $\widetilde{H_j}$ be the largest open interval (or one of the two in case of non-uniqueness) satisfying $H_j\cap \widetilde{H_j}\neq \emptyset$ and $\widetilde{H_j}\subseteq F_{<c_{\ell}}\setminus F$, where $F_{<c_{\ell}}\defeq \{x\in\mathbb R\mid \inf_{a\in F}\lvert x-a\rvert <c_{\ell} \}$.
By construction, $\{\widetilde{H_j} \}_{j\in J_{\ell}}$ is a pairwise disjoint family (here it is important to restrict to $j\in J_{\ell}$).
Furthermore, $\lambda(\widetilde{H_j})=c_{\ell}$. Therefore,
\begin{align*}
	\# J_{\ell}
	= \tfrac{1}{c_{\ell}}\sum_{j\in J_{\ell}} \lambda(\widetilde{H_j})
	=\tfrac{1}{c_{\ell}}\lambda\Big(\bigcup_{j\in J_{\ell}} \widetilde{H_j} \Big)
	\leq \tfrac{1}{c_{\ell}}\lambda(F_{c_{\ell}})
	<\infty
\end{align*}
and whence $\# J =\sum_{\ell=1}^{k} \# J_{\ell}<\infty$, showing \ref{it:claimI}.

Our second step is to verify that
\begin{enumerate}[label=(\Roman*),start=2]
	\item \label{claimII} the number of connected components of $U\setminus \Phi U$ is finite.
\end{enumerate}
For this, note that the family of connected components of $U\setminus \Phi U$ essentially coincides with $\{H_j\}_{j\in J}$, with the only possible differences occurring at the boundary points $\bigcup_{j\in J}\partial H_j$. More precisely, each $H_j$ is a subset of $U\setminus \Phi U$ but not necessarily a connected component of this set, and $U\setminus \Phi U\subset \bigcup_{j\in J} \overline{H}_j$. Therefore, $U\setminus \Phi U$ has at most $\#J$ connected components and \ref{it:claimI} implies \ref{claimII}.

For the third step let $U_{i_1},\ldots, U_{i_n}$ denote those connected components of $U$ which intersect $U\setminus \Phi U$ and let $E^*\defeq\{i_1,\ldots,i_n \}$ denote the respective index set. (The finiteness of $E^*$ is clear from assertion \ref{claimII}.) We prove
\begin{enumerate}[label=(\Roman*),start=3]
	\item \label{claimIV} $\bigcup_{j\in E^*} U_{j}\cap F\neq\emptyset$.
\end{enumerate}
If $i\in E\setminus E^*$ then $U_i\subseteq \Phi U=\bigcup_{j,k} \phi_k(U_j)$, which is a disjoint union by OSC and definition of the $U_j$. As furthermore $U_i$ and each $\phi_k(U_j)$ is open and connected there exist $k\in\Sigma$, $j\in E$ so that $U_i\subset\phi_k(U_j)$. On the other hand, since $\phi_k(U_j)$ is a connected subset of $U$ and intersects the connected component $U_i$, it must be contained in $U_i$. Thus $U_i=\phi_k(U_j)$, i.\,e.\ $U_i$ is the precise image of $U_j$ under $\phi_k$.
Amongst $\{U_i\mid i\in E\setminus E^*\}$ there is at least one largest bounded one, $U_{i^*}$, which needs to be an image of some $U_{j}$ with $j\in E^*$ by the contraction property of the $\phi_k$. Amongst $\{U_i\mid i\in E\setminus E^*\}\setminus \{U_{i^*}\}$ there again is a largest bounded one and inductively we see that each bounded $U_i$ with $i\in E\setminus E^*$ is the image of one of the sets $U_{j}$, $j\in E^*$ and that also the possible unbounded components need to be amongst $U_{i_1},\ldots, U_{i_n}$.
Since $U$ is strong, we conclude $\bigcup_{j\in E^*} U_{j}\cap F\neq\emptyset$, showing \ref{claimIV}.

In the final step we deduce \eqref{claimIII} from the above:
By \ref{claimIV}, we can find a minimal $m\in\mathbb N$ and $\om\in \Sigma^m$ so that $\phi_{\om}I\subset \bigcup_{j\in E^*} U_{j}$. Let $\om=\om_1\cdots\om_m$.
If $\phi_{\om_2\cdots\om_m}I$ intersected infinitely many of the positively separated $U_i$, then $U\setminus\Phi U$ would have infinitely many connected components. 
(Assume that $\phi_{\om_2\cdots\om_m}I$ intersects infinitely many $U_i$, say $U_{j_1}, U_{j_2},...$. Then on the one hand,  $\phi_\om I$ intersects all the sets $\phi_{\om_1} U_{j_k}$, i.\,e.\ infinitely many. On the other hand, $\phi_\om I$ is an interval and thus contained in a connected component $V$ of $U$. Thus, as each $\phi_{\om_1}U_{j_k}$ is connected, $\phi_{\om_1}U_{j_k}\subset V$. Further, the family $\{ \phi_{\om_1} U_{j_k}\}_k$ is disjoint, which implies that the set $V\setminus \phi_{\om_1}U$ and thus $V\setminus \Phi U$ and $U\setminus \Phi U$ must have infinitely many connected components.) 
This contradicts assertion \ref{claimII}. Thus, $\phi_{\om_2\cdots\om_m}I$ intersects only finitely many $U_i$. Inductively, we see that $I$ can only intersect finitely many $U_i$, whence $\#\tilde E<\infty$, proving \eqref{claimIII}.
\end{proof}

\bibliographystyle{alpha}
\bibliography{Referenzen}

\begin{thebibliography}{KPW16}

\bibitem[BM09]{BandtMesing}
C.~Bandt and M.~Mesing.
\newblock Self-affine fractals of finite type.
\newblock {\em Banach Center Publications}, 84(1):131--148, 2009.

\bibitem[BMT17]{Bandt_pinwheel}
C.~Bandt, D.~Mekhontsev, and A.~Tetenov.
\newblock A single fractal pinwheel tile.
\newblock {\em Proc. Amer. Math. Soc.}, 2017.
\newblock Published electronically.

\bibitem[Fal95]{Falconer_Minkowski}
K.~J. Falconer.
\newblock On the {M}inkowski measurability of fractals.
\newblock {\em Proc. Amer. Math. Soc.}, 123(4):1115--1124, 1995.

\bibitem[Gat00]{Gatzouras}
D.~Gatzouras.
\newblock Lacunarity of self-similar and stochastically self-similar sets.
\newblock {\em Trans. Amer. Math. Soc.}, 352(5):1953--1983, 2000.

\bibitem[HL08]{HeLau}
X.-G. He and K.-S. Lau.
\newblock On a generalized dimension of self-affine fractals.
\newblock {\em Math. Nachr.}, 281(8):1142--1158, 2008.

\bibitem[KK15]{Euler}
M.~Kesseb\"ohmer and S.~Kombrink.
\newblock Minkowski content and fractal {E}uler characteristic for conformal
  graph directed systems.
\newblock {\em J. Fractal Geom.}, 2(2):171--227, 2015.

\bibitem[KPW16]{ErinSteffen}
S.~Kombrink, E.~P.~J. Pearse, and S.~Winter.
\newblock Lattice-type self-similar sets with pluriphase generators fail to be
  {M}inkowski measurable.
\newblock {\em Math. Z.}, 283(3-4):1049--1070, 2016.

\bibitem[Lap93]{Lapidus_Dundee}
M.~L. Lapidus.
\newblock Vibrations of fractal drums, the {R}iemann hypothesis, waves in
  fractal media and the {W}eyl-{B}erry conjecture.
\newblock In {\em Ordinary and partial differential equations, {V}ol.\ {IV}
  ({D}undee, 1992)}, volume 289 of {\em Pitman Res. Notes Math. Ser.}, pages
  126--209. Longman Sci. Tech., Harlow, 1993.

\bibitem[LP93]{LapidusPomerance}
M.~L. Lapidus and C.~Pomerance.
\newblock The {R}iemann zeta-function and the one-dimensional {W}eyl-{B}erry
  conjecture for fractal drums.
\newblock {\em Proc. London Math. Soc. (3)}, 66(1):41--69, 1993.

\bibitem[LvF00]{LapidusvanFrankenhuijsen}
M.~L. Lapidus and M.~van Frankenhuysen.
\newblock {\em Fractal geometry and number theory}.
\newblock Birkh\"auser Boston, Inc., Boston, MA, 2000.
\newblock Complex dimensions of fractal strings and zeros of zeta functions.

\bibitem[Man95]{Mandelbrot}
B.~B. Mandelbrot.
\newblock Measures of fractal lacunarity: {M}inkowski content and alternatives.
\newblock In {\em Fractal geometry and stochastics ({F}insterbergen, 1994)},
  volume~37 of {\em Progr. Probab.}, pages 15--42. Birkh\"auser, Basel, 1995.

\bibitem[PW12]{ErinSteffen_Tilings}
E.~P.~J. Pearse and S.~Winter.
\newblock Geometry of canonical self-similar tilings.
\newblock {\em Rocky Mountain J. Math.}, 42(4):1327--1357, 2012.

\bibitem[Sch94]{Schief}
A.~Schief.
\newblock Separation properties for self-similar sets.
\newblock {\em Proc. Amer. Math. Soc.}, 122(1):111--115, 1994.

\bibitem[Sta76]{Stacho}
L.~L. Stach{\'o}.
\newblock On the volume function of parallel sets.
\newblock {\em Acta Sci. Math. (Szeged)}, 38(3--4):365--374, 1976.

\bibitem[Win15]{W15}
S.~Winter.
\newblock Minkowski content and fractal curvatures of self-similar tilings and
  generator formulas for self-similar sets.
\newblock {\em Adv. Math.}, 274:285--322, 2015.

\end{thebibliography}

\end{document}